\documentclass[11pt]{article}
\usepackage[left=3cm,right=3cm,
    top=3cm,bottom=3cm,bindingoffset=0cm]{geometry}
\usepackage{graphicx}
\usepackage[english]{babel}
\usepackage{amsmath,amscd, wasysym, mathtools,array, mathtools}
\usepackage{upgreek} 
\usepackage{amssymb,amsthm,euscript}
\usepackage{xcolor,graphicx,epsfig,subfigure}
\usepackage{dsfont}
\usepackage{enumitem}

\usepackage{amsfonts}
\usepackage{fancyhdr}
\newtheorem{theorem}{Theorem}[section]
\newtheorem*{theorem*}{Theorem}
\newtheorem{lemma}[theorem]{Lemma}
\newtheorem*{lemma*}{Lemma}

\newtheorem{claim}[theorem]{Claim}

\newtheorem{conjecture}[theorem]{Conjecture}

\theoremstyle{definition}

\newtheorem{remark}[theorem]{Remark}

\renewenvironment{proof}{
\par\noindent{\it Proof.}} {\mbox{}\hfill$\blacksquare$ \par }

\newcommand{\bbP}{\mathbb{P}}
\newcommand{\bbA}{\mathbb{A}}
\newcommand{\bbG}{\mathbb{G}}

\newcommand {\Gal}{\mathord{\rm {Gal}}}

\newcommand{\Pic}{\mathord{\rm{Pic}}}

\newcommand{\id}{\mathord{\rm{id}}}

\newcommand{\sm}{\mathord{\rm{sm}}}
\newcommand{\sing}{\mathord{\rm{sing}}}
\newcommand{\rk}{\mathord{\rm{rk}}}

\DeclareFontFamily{U}{wncy}{}
    \DeclareFontShape{U}{wncy}{m}{n}{<->wncyr10}{}
    \DeclareSymbolFont{mcy}{U}{wncy}{m}{n}
    \DeclareMathSymbol{\Sh}{\mathord}{mcy}{"58}

\title{\LARGE{\scshape{On the arithmetic of intersections of\\ two quadrics containing a conic}}}
\author{Per Salberger}
\date{}

\begin{document}

\maketitle

\let\thefootnote\relax\footnotetext{This is a retyped version of an unpublished typescript dated 1993. I wrote that paper  in 1993.
It was not submitted for publication.
 In the last two years,  the paper was used by J. Iyer and R. Parimala (https://arxiv.org/abs/2201.12780), by J.-L. Colliot-Th\'el\`ene (https://arxiv.org/abs/2208.04121)
and by A. Molyakov (https://arxiv.org/abs/2305.00313). 
 It thus seems reasonable to make that  paper available.}

\section{Introduction}

The aim of this paper is to prove the following two conjectures of Colliot-Th\'el\`ene, Sansuc and Swinnerton-Dyer \cite[\S16]{CSS} in the case of intersections of two quadrics containing a conic.  

\begin{conjecture}\label{01}
Let $k$ be a number field and $X\subset \bbP^n_k$, $n\geq 6$ be a non-conical geometrically integral intersection of two quadrics. Suppose that there exists a smooth $k_v$-point on $X$ for each place $v$ of $k$. Then there exists a smooth $k$-point on $X$.  
\end{conjecture}

\begin{conjecture}\label{02}
Let $k$ be a number field and $X\subset \bbP^5_k$ be a smooth intersection of two quadrics. Suppose that there exists a $k_v$-point on $X$ for each place $v$ of $k$. Then there exists a $k$-point on $X$.
\end{conjecture}

The paper of Colliot-Th\'el\`ene, Sansuc, Swinnerton-Dyer \cite{CSS} contains a proof of these conjectures in the following cases:

\begin{itemize}
    \item[(a)] $X$ contains two conjugate singular points
    \item[(b)] $X$ contains two conjugate skew lines
    \item[(c)] $X$ contains a two-dimensional quadric
    \item[(d)] $n\geq 8$
\end{itemize}

The authors of \cite{CSS} used a trick of Mordell suggested to them by Heath-Brown to  deduce (d) from the case where the intersection of two quadrics contains a conic. It is thus interesting to consider intersections of quadrics of lower dimension containing
a conic. The basic idea of the proof of (d) in \cite[\S10]{CSS} is to
construct hyperplane sections containing the conic and smooth $k_v$-points at each place $v$ of $k$. By choosing these hyperplanes with great care it is possible to reduce to a special class of singular intersections of two quadrics in $\bbP_k^4$ containing a conic.
The authors of \cite{CSS} then complete their proof by applying the deep results for Ch\^{a}telet surfaces obtained in the previous sections (cf. also \cite{Sb2}, \cite{Sb3} for other proofs of the results on Ch\^{a}telet surfaces). But the method just described does not seem to be sufficiently powerful to prove the following result.

\begin{theorem}\label{03}
Let $k$ be a number field and $X\subset \bbP^n_k$ $(n\geq 6)$ be a non-conical geometrically integral intersection of two quadrics containing
a conic. Suppose that there exists a smooth $k_v$-point on $X$ at each place $v$ of $k$. Then there exists a smooth $k$-point on $X$.
\end{theorem}

\begin{theorem}\label{04}
Let $k$ be a number field and $X\subset \bbP^5_k$ be a smooth intersection of two quadrics containing a conic. Suppose that there exists a $k_v$-point on $X$ for each place $v$ of $k$. Then there exists a $k$-point on $X$.
\end{theorem}

We note that the statements in theorem \eqref{03} respectively \eqref{04} are false for intersections of quadrics in $\bbP^5_k$ respectively $\bbP^4_k$. The main new ingredient in the proofs of \eqref{03}, \eqref{04} is the following deep result of the author. 

\begin{theorem}\label{05}
Let $k$ be a number field and $Y\subset \bbP^4_k$ be a smooth intersection of two quadrics containing a conic such that the functorial map $$H^2_{\text{\'et}}(k,\bbG_m) \to H^2_{\text{\'et}}(Y,\bbG_m)$$ is surjective. Suppose that there exists a $k_v$-point on $Y$ for each place $v$ of $k$. Then there exists a $k$-point on $Y$.
\end{theorem}

\begin{proof}
    If $C\subset Y$ is a (smooth) conic, then the linear system of $C$ defines a conic bundle morphism $Y\to \bbP^1_k$ (cf. e.g. \cite{Is}). We may thus apply the Hasse principle for conic bundle surfaces established in \cite{Sb1} (cf. also \cite{Sb2}). For another proof which is also due to the author (letter to Skorobogatov 27/3/87) cf. the report by Colliot-Th\'el\`ene \cite{Co}.
\end{proof}

To prove \eqref{03} and \eqref{04} we consider as in \cite{CSS} $k$-hyperplane
sections containing the given conic and smooth $k_v$-points for each
place $v$ of $k$. It is quite simple to construct such hyperplane sections, but some work is needed to obtain geometrically integral hyperplane sections of this type and some degenerate cases have to be treated by other methods. To prove \eqref{03} for $n=6$, it is necessary to first generalize \eqref{04} to a wider class of intersections of two quadrics in $\bbP^5_k$ including some singular varieties. When $n=5$, we must construct hyperplane sections with trivial Brauer group, in order to apply \eqref{05}. To achieve this we use a strong version of Hilbert's irreducibility theorem.

\section{Intersections of two quadrics in $\bbP^4$}

To start the induction, we will use the following result.

\begin{theorem}\label{11}
    
Let $k$ be a number field and $F$, $G$ be two quadratic forms in $X_0,X_1,\dots,X_4$ with coefficients in $k$ such that $F(X_0,X_1,X_2,0,0)$ is of rank $3$ and $G(X_0,X_1,X_2,0,0)$ is the zero form. Suppose further that $F$ is of rank $5$, $G$ is of rank greater of equal to $3$ and that the discriminant polynomial $$P(\lambda)=\det(F+\lambda G)$$ is irreducible. Then the $k$-subscheme
$Y\subset \bbP^4_k$ defined by $F=G=0$ satisfies:
\begin{itemize}
    \item[(i)]  $Y$ is geometrically integral and of codimension 2 in $\bbP^4_k$.
    \item[(ii)] If $V$ is a smooth proper $k$-model of $Y$, and $\bar k$ an algebraic closure of $k$, then 
    $$H^1\big(\Gal(\bar k/k),\Pic (V\times \bar k)\big)=0$$ 
    and the functorial map 
    $$H^2_{\text{\'et}}(k,\bbG_m) \to H^2_{\text{\'et}}(Y,\bbG_m)$$
    is surjective.
    \item[(iii)] $Y$ satisfies the smooth Hasse principle and weak approximation.
\end{itemize}
\end{theorem}

\begin{proof}
    \begin{itemize}
        \item[(i)] This is a special case of \cite[1.11]{CSS}
        \item[(ii)] This is clear from \cite[3.19]{CSS} if the following hypothesis does not hold:
            
(E) \textit{In the pencil $\mu F+\lambda G\; (\mu,\lambda \in \bar k)$, there exists exist a pair of non-proportional forms of rank $4$ both defined over $k$ or both defined over a quadratic extension $K$ of $k$ and conjugate under the Galois group $\Gal(K/k)$.}

But if (E) were true, then the homogeneous quintic polynomial
$$D(\mu,\lambda)=\mu^5 P(\lambda/\mu)\in k[\mu,\lambda]$$ would contain a quadratic factor in $k[\mu,\lambda]$ with two different zeroes in $\bbP^1_{\bar k}$, which is impossible since $P(\lambda)$ is irreducible and the rank of $G$ is greater or equal to 3.

\item[(iii)] If the singularities of $Y$ are not isolated, then this is proved in \cite[3.15(iv)]{CSS}. If $Y$ is singular with isolated singularities and not an Iskovskih surface, then one may find a proof in the paper of Coray and Tsfasman \cite{CrTs}. If $Y$ is an lskovskih surface (i.e. if $Y$ contains exactly two singular points and these are not rational over $k$), then the result follows from (ii) and \cite[9.7]{CSS}. Finally, if $Y$ is smooth, then the Hasse principle follows from (ii) and \eqref{05}, while weak approximation is treated in \cite{CoSk} and \cite{SbSk}.
\end{itemize}
\end{proof}

\section{Intersections of quadrics in $\bbP^5$}

\begin{theorem}\label{21}
    Let $k$ be a number field and $\bar k$ be an algebraic closure of $k$. Let $F$, $G$ two quadratic forms in $x_0,x_1,\dots,x_5$ with coefficients in $k$ such that $F(x_0,x_1,x_2, 0, 0 ,0)$ is of rank $3$ and $G(x_0, x_1, x_2, 0, 0 ,0)$ is the zero form. Suppose that the following conditions are satisfied:
\begin{itemize}
    \item[(i)] $\rk(F)=6$
    \item[(ii)] $\rk(F+\lambda G)\geq 5$ for each $\lambda \in \bar k$
    \item[(iii)] $\rk(G)\geq 3$ 
\end{itemize}
Then the subscheme $X\subset \bbP^5_k$ defined by $F=G=0$ is a pure and geometrically integral intersection satisfying the smooth Hasse principle and weak approximation.
\end{theorem}

We first note that the result is known if $G$ is of rank less of equal to $3$ \cite[3.14]{CSS}. If $\rk(G)=4$, the quadric $G=0$ is a cone with a singular line $\ell$. All points of the intersection $\{F=0\}\cap \ell$ are singular points of $X$. This intersection possesses a $k$-point or a pair of conjugate points defined over a quadratic extension $K\supset k$. In the first case $X$ is $k$-birational to a geometrically integral quadric in $\bbP^4_k$ by \cite[2.1]{CSS} which implies the result. In the second case the result follows from \cite[9.1]{CSS} since the the special case (E) is not possible as $G$ is the only quadric of rank $4$ in the pencil $\mu F+\lambda G\; (\mu,\lambda \in \bar k)$.  We may therefore replace (iii) by the stronger condition
\begin{equation}\tag{2.2}\label{22}
    \rk(G)\geq 5
\end{equation}
Next, let $\Pi=\bbP^2_k$ be the $k$-plane defined by $x_3=x_4=x_5=0$. If $P$ is a point on $X\backslash \Pi$, let $\langle P,\Pi\rangle$ be the $3$-dimensional linear subspace
generated by $P$ and $\Pi$. The following result is true for a $k$-variety as in \eqref{21} for any perfect field $k$.

\setcounter{theorem}{2}

\begin{lemma}\label{23}
There exists a Zariski open subset $X_0\subset X$ such that for any $k$-point $P\in X_0\backslash \Pi$ there exists a hyperplane  $H\supset \Pi$ defined over $k$ such that $P$ is a smooth point on the intersection $X\cap H$.
\end{lemma}
\begin{proof} We need the following result:
\begin{theorem*}[Bertini]
Let $X$ be a smooth quasi-projective variety over a field $k$ of characteristic 0 and let 
 $\mathcal{L}$ be a linear system of divisors on $X$. Assume that $\mathcal{L}$ has no base points. Then there is a Zariski open subset $U\subset |\mathcal{L}|$ such that any divisor $D\in U$ is smooth.
    
\end{theorem*}

Let $\mathcal{L}\subset \Gamma\big(\bbP^5_k,\mathcal{O}(1)\big)=\Gamma\big(X,\mathcal{O}(1)\big)$ be the linear system on $X$ consisting of linear forms which vanish on $\Pi$. The set of base points of $\mathcal{L}$
is $X\cap \Pi$. Then $\mathcal{L}$ induces a base-point-free linear system on the smooth quasi-projective variety $X_{\sm}\backslash \Pi$. By Bertini's theorem we deduce that the intersection $(X_{\sm}\backslash \Pi)\cap H$ is smooth for hyperplanes $H$ in an open subset $U \subset \bbP(\mathcal{L})$. Let $\phi\colon X_{\sm}\backslash \Pi \to \bbP(\mathcal{L})$ be the natural map and let $X_0=\phi^{-1}(U)$. Then $X_0\subset X_{\sm}\backslash \Pi$ is open and for any point $P \in X_0$ there exists a hyperplane $H\supset \Pi$ which intersects $X$ transversally at $P$.
\end{proof}

\begin{remark}\label{24}
Let $U=X_0\backslash \Pi$. Then $U$ is Zariski-open and non-empty. This implies by the implicit function theorem \cite[164]{CCS} that $U(k_v)$ is dense in $X_{\sm}(k_v)$ in the $v$-adic topology for any place $v$ of $k$.
\end{remark}

Let $V= \bbP^2_k$ be the $k$-variety parametrizing the hyperplanes containing $\Pi$ and $Z\subset \bbA^1_k\times V$ be the incidence variety consisting of pairs $(Q_{\lambda}, H)$ of a quadric $$Q_{\lambda}:\;F+\lambda G=0$$ and a hyperplane $H\supset \Pi$ for which $Q_{\lambda}\cap H$ is singular. Thus $Z\subset \bbA^1_k\times V$ is defined by the equation $$\det\big((F+\lambda G)|_H\big)=0.$$

\begin{claim}\label{25}
 The variety $Z$ does not coincide with $\bbA^1_k\times V$ and has only one irreducible component of dimension $2$.
\end{claim}

\begin{proof}
 Let $\lambda\in\bar k$ and $Q_{\lambda}$ be the quadric $F+\lambda G=0$. Then, since $Q_{\lambda}\cap \Pi$ is smooth and $\rk (F+\lambda G)\geq 5$, there is a hyperplane $H\supset\Pi$ such that $Q_{\lambda}\cap H$ is smooth. Hence the projection from $Z\subset \bbA^1_k\times V$ to $\bbA^1_k$ does not contain any fibre of the projection $\bbA^1_k\times V \to \bbA^1_k.$
 It is thus sufficient to prove that the generic fibre of the projection map from $Z$ to $\bbA^1_k$ is irreducible. But the generic form in the pencil defines a smooth quadric $Q\subset\bbP^5_{k(\lambda)}$ which intersects $\Pi$ in a smooth conic. This implies that the dual variety $\tilde Q$ of $Q$ is a smooth quadric in the dual projective space $\tilde \bbP^5_{k(\lambda)}$ which intersects $V\subset \tilde \bbP^5_{k(\lambda)}$ in a smooth conic. But
$\tilde Q\cap V$ is nothing but the generic fibre of $Z \to \bbA^1_k$.

\end{proof}

We now return to the proof of \eqref{21}. We first note that the projection 
$p\colon Z\to V$
is surjective since for any hyperplane $H\in V(\bar k)$ there exists $\lambda\in\bar k$ such that $Q_{\lambda}\cap H$ is singular. From \eqref{25} it is clear that that there exists a non-empty open $k$-subvariety $V_0\subset V=\bbP^2_k$ for which $Z_0=p^{-1}(V_0)$ is irreducible. Since the morphism $p$ is generically finite, there exists a sufficiently small non-empty subset $V_0\subset V$ such that the restriction $p_0\colon Z_0\to V_0$ of $p\colon Z\to V$ is finite and \'etale.

We may then apply the version of Hilbert's irreducibility theorem
described in \cite{Ek} to $p_0$ and conclude:

\begin{claim}\label{26}
The subset of $k$-points in $V_0(k)$ with irreducible inverse
image under $p_0$ is dense in $\displaystyle\prod_{v\in S} V_0(k_v)$ and hence in $\displaystyle\prod_{v\in S} V(k_v)$ for any finite set $S$ of places of $k$.
\end{claim}

Now let $C$ be the smooth $k$-conic $X\cap \Pi$ and $S$ be a finite set of places containing each place $v$ for which $C(k_v)$ is empty. Suppose that for each $v\in S$, we are given a smooth $k_v$-point $P_v$ on $X$ and an open $v$-adic neighbourhood $N_v\subset V_{\sm}(k_v)$ of $P_v$. Let us show that there exists
a smooth $k$-point $P$ on $X$ such that $P\in N_v$ for each $v\in S$. For each $v\in S$ there is (cf. \eqref{24}, \eqref{23}) a $k_v$-point $R_v\in N_v$ and a $k_v$-hyperplane $H_v\in V(k_v)$ containing $R_v$ such that $H_v$ intersects $X$ transversally at $R_v$. The implicit function theorem \cite[6.2]{CSS}
implies that if for each $v\in S$ there exists a $v$-adic neighbourhood
$O_v\subset V(k_v)$ of $H_v\in V(k_v)$ such that any $H\in O_v$ intersects X transversally in a $k_v$-point of $N_v$. Now apply \eqref{26} and choose $H\in V_0(k)$ with $p^{-1}(H)$ irreducible and such that $H\in O_v$ for each $v\in S$. Then,
$(X\cap H)_{\sm}(k_v)\cap N_v$ is non-empty for each $v\in S$. Further, since $X\cap H$ contains $C=X\cap \Pi$, we have that $(X\cap H)_{\sm}(k_v)$ is non-empty for each $v\notin S$. We therefore obtain the desired conclusion in the case $X\cap H$ satisfies the hypothesis of \eqref{11}.

We may assume that $H$ is not tangent to the quadric $F=0$, since the set of set of such $H$ form an open subset of $V$ and \eqref{26}
still holds if we replace $V_0$ by a smaller open non-empty subset of $F$. We have then that the restriction of $F$ to the affine cone of $H$ is of maximal rank $5$. We have also by \eqref{22} that the restriction of $G$ to the affine cone of $H$ is of rank greater of equal to 3.

Finally, the condition that $p_0^{-1}(H)$ is irreducible means that the discriminant polynomial $P(\lambda)$ of the restriction of $F+\lambda G$ to the affine cone of $H$ is irreducible. We may therefore apply \eqref{11} to $X\cap H$ and find a smooth $k$-point on $X\cap H$ which belongs to $N_v$ for
each $v\in S$. This completes the proof of \eqref{21}.

\begin{theorem}
 Let $k$ be a number field and $X\subset \bbP^5_k$ be a smooth intersection of two quadrics containing a smooth $k$-conic. Then $X$
satisfies the Hasse principle and weak approximation.   
\end{theorem}

\begin{proof}
It follows from the smoothness of $X$ \cite[1.13]{CSS} that each
quadratic form in the $k$-pencil defining $X$
 is of rank greater of equal to 5 and that the generic form is of rank 6. Let $\Pi$ be the plane such that $X\cap \Pi=C$ is the $k$-conic on $X$. Let $F$, $G$ be two quadratic $k$-forms in the pencil defining $X$ such that $F$ is of rank 6, the restriction of $F$ to $\Pi$ is of rank 3 and $G$ vanishes at some point $S\in \Pi\backslash C$. Since $G$ also vanishes on $C$ it vanishes on the whole plane $\Pi$. If we now choose $k$-coordinates $x_0,x_1,\dots,x_5$ such that $\Pi$ is given by $x_3=x_4=x_5=0$, then $X$ satisfies all the hypothesis of \eqref{21}.
\end{proof}

\section{Intersections of two quadrics in $\bbP^n,\;n\geq 6$}

\begin{theorem}\label{31}
Let $k$ be a number field and $\bar k$ be an algebraic closure of $k$. Let $F$, $G$ be  two quadratic forms in $x_0,x_1,\dots,x_n$ where $n\geq 5$
with coefficients in $k$ such that $F(x_0,x_1,x_2,0,\dots,0)$ is of rank $3$ and $G(x_0,x_1,x_2,0,\dots,0)$ is the zero form. Suppose that the following conditions are satisfied:
\begin{itemize}
    \item[(i)] $F$ is of rank $n+1$
    \item[(ii)] $F+\lambda G$ is of rank greater or equal to $5$ for each $\lambda\in\bar k$
    \item[(iii)] $G$ is of rank greater or equal to $3$
\end{itemize}
Then the subscheme $X\subset \bbP^k_n$ defined by $F=G=0$ is a pure and geometrically integral intersection satisfying the smooth Hasse principle and weak approximation.
\end{theorem}
\begin{proof}
The case $n=5$ has already been treated, so we may and shall assume that $n\geq 6$ and that the result has already been proved in lower dimensions. We shall also use the fact that the elementary case when $\rk(G)\leq n-2$ is known \cite[3.14]{CSS} and assume that:
\begin{equation}\tag{3.2}\label{32}
\rk(G)\geq n-1
\end{equation}

Let $\Pi\subset \bbP^k_n$ be the $k$-plane defined by $x_3=x_4=\dots=x_n=0$ and $V=\bbP^{n-3}_k$ be the $k$-variety parametrizing the hyperplanes containing $\Pi$. Let us fix a singular $\bar k$-point $M_i$ on each singular quadric in the pencil $F+\lambda G=0\;(\lambda\in\bar k)$. Let
$V_0\subset V$ be the open $k$-subvariety of hyperplanes 
\begin{equation}\tag{$\star$}\label{x}
    \sum_{i=3}^n \alpha_i x_i =0
\end{equation}
such that \eqref{x} does not contain any of the finite set of points $M_i$ and \eqref{x} is not tangent to the quadric $F=0$ i.e. the restriction of $F$ to \eqref{x} is of rank $n$. It is easy to see that $V_0$ is non-empty. Suppose that $n\geq 6$. Then, if (i), (ii) and \eqref{32} hold  for $X$ and $H\in V_0(k)$ then (i), (ii) and (iii) hold for the restrictions of the quadratic forms to the affine cone of $X\cap H$ \cite[1.16]{CSS}. We have thus by the induction assumption that the smooth Hasse principle and weak approximation hold for $X\cap H$ if $H\in V_0(k)$.

Now let $C$ be the smooth $k$-conic $X\cap H$ and $S$ be a finite set of places containing each place $v$ for which $C(k_v)$ is empty. Suppose that for each $v\in S$, we are given a smooth $k_v$-point $P_v$ on $X$ and an open neighbourhood $N_v\subset V_{\sm}(k_v)$ of $P_v$. Let us show that there exists a smooth $k$-point $P$ on $X$ such that $P\in N_v$ for each $v\in S$. The proof is similar but simpler than for $n=5$. The following result will replace lemma \eqref{23}.

\setcounter{theorem}{2}
\begin{claim}\label{33}
Let $X\subset \bbP^n_k,\; n\geq 6$ and $\Pi$ be as above and $K\supset k$ be an extension of fields. Then, if $P$ is a smooth $K$-point $X_K$, then there exists a $K$-hyperplane $H\supset\langle P, \Pi\rangle$ intersecting $X_K$ transversally in $P$.
\end{claim}
\begin{proof}
    Just choose a hyperplane $H\supset\langle P, \Pi\rangle$ which does not contain the tangent space of $X$ at $P$.
\end{proof}

We now complete the proof of \eqref{31}. Let $H_v\in V(k_v)$ be $k_v$-hyperplanes containing $P_v$ such that $H_v$ intersects $X$ transversally in $P_v$ for each $v\in S$. We may then repeat the implicit function argument used in the proof of \eqref{21} and find $v$-adic neighbourhoods $O_v\subset V(k_v)$ of $H_v\in V(k_v)$
such that any $H\in O_v$ intersects $X$ transversally at a $k_v$-point in $N_v$. Now use the fact that $V_0(k)$ is dense in $\prod_{v\in S} {V(k_v)}$ and apply \eqref{21} to $X\cap H$ for some $H\in V_0(k)$ such that $H\in O_v$ for each $v\in S$. This completes the proof of \eqref{31}.

\end{proof}

\begin{theorem}\label{34}
 Let $k$ be a number field and $X\subset \bbP^n_k,\; n\geq 6$ be a pure non-conical geometrically integral intersection of two quadrics. Suppose that $X$ contains a smooth $k$-conic $C$. Then the smooth Hasse principle and weak approximation hold for $X$.
\end{theorem}

\begin{proof}
    We may introduce $k$-coordinates $x_0,x_1,\dots,x_n$ such that $X$ is defined by two quadratic forms $F(x_0,x_1,\dots x_n)$ , $G(x_0,x_1,\dots x_n)$ for which $F(x_0,x_1,x_2,0,\dots,0)$ is of rank 3 and $F(x_0,x_1,x_2,0,\dots,0)$ is the zero form. We may also assume that $P(\lambda)=\det(F+\lambda G)$ does not vanish identically since otherwise $X_{\sing}(k)$ is non-empty and the conclusion follows from \cite[2.1]{CSS}. But if $P(\lambda)$ is non-trivial, then $F+\lambda G$ is of maximal rank for some $\lambda \in k$ and we may assume that $\rk(F) =n+1$ by a linear change of coordinates in the pencil of quadrics.

    Recall that \cite[1.15]{CSS} $P$ and $\lambda \in \bar k$ has a zero of multiplicity at least $n+1-\rk(F+\lambda G)$. Thus if $\rk(F+\lambda G)\leq 4$ for $s$ different values $\lambda\in\{\lambda_1,\lambda_2,\dots,\lambda_s\}$ then 
    $$
    s(n+1-4)\leq \deg(P)\leq n+1
    $$
    and hence 
    $$
    s\Big(1-\frac{4}{7}\Big)\leq s\Big(1-\frac{4}{n+1}\Big)\leq 1.
    $$
    
    This implies that $s\leq 2$ with equality only if $n+1=7$ or $n+1=8$. If $s=0$ i.e. if $\rk(F+\lambda G)\geq 5$ for each $\lambda \in \bar k$, then we obtain the desired conclusion from \eqref{31} and if $s=1$, then this $\lambda_1$ belongs to $k$ and we may apply 
    \cite[3.14]{CSS}. So suppose that $\rk(F+\lambda G)\leq 4$ for exactly two $\lambda \in \bar k$, say $\lambda_1$ and $\lambda_2$. If $n+1=8$ then $P$ has two zeroes $\lambda_1$ and $\lambda_2$ of multiplicity 4 and hence no further zeroes. If $n+1 =7$, then $P$ has two zeroes $\lambda_1$ and $\lambda_2$ of multiplicity at least 3 and hence at most one further zero which must then be a simple zero which belongs to $k$. There are thus in both cases $n+1=7$ or $n+1=8$ only two possibilities. Either $\lambda_1$ and $\lambda_2$ belong to $k$ or $\lambda_1$ and $\lambda_2$ belong to a quadratic extension $K\supset k$ and are conjugate under $\Gal(K/k)$. If $\lambda_1$ and $\lambda_2$ belong to $k$, then we may apply the elementary result
    \cite[3.14]{CSS} once more. So suppose that $\lambda_1$ and $\lambda_2$ do not belong to $k$. The case $n=7$ is easily deduced from the case $n=6$ by considering hyperplane sections containing a smooth conic and applying \eqref{33} and the arguments in the end proof of \eqref{31}.

    We are therefore left with varieties in $\bbP^6_k$ defined by two quadratic forms $F+\lambda_i G$ with $\lambda_i\in K\backslash k$ of rank 4 which are conjugate under $\Gal(K/k)=\{\id, \sigma\}$. The singular loci of the quadrics $F+\lambda_i G$ are then a pair of disjoint skew planes conjugate under $\Gal(K/k)$. Let us choose $K$-coordinates $x_0,x_1,\dots,x_6$ such that the coordinate $x_6$ is defined over $k$, the first plane is given by $x_3=x_4=x_5=x_6=0$ and the second one is given by $x_0=x_1=x_2=x_6=0$. Then the first quadric is given by a quadratic form $T(x_3,x_4,x_5,x_6)$ and the second one by the conjugate form $T^{\sigma}(x_0,x_1,x_2,x_6)$. Let $H\subset \bbP^6_k$ be the $k$-hyperplane $x_6=0$ and $X_0=X\backslash H$. Then $X_0$ is a smooth affine $k$-subvariety of $\bbA^6_k=\bbP^6_k\backslash H$ which is defined by equations $T(x_3,x_4,x_5,1)=0$, $T^{\sigma}(x_0,x_1,x_2,1)=0$, so it can be decomposed into a product $X_0=Q\times Q^{\sigma}$ where $Q\subset \bbA^3_K$ is the smooth quadric given by $T(x_3,x_4,x_5,1)=0$. So $X_0$ is the Weil restriction $R_{K/k}(Q)$ \cite[pp. 191-195]{BLR}. Therefore the Hasse principle
    and weak approximation hold for $X_0$ if and only if they hold for $Q$. But the Hasse principle for $Q$ follows from Hasse's theorem on quadratic forms and weak approximation from the fact that a $K$-quadric with a $K$-point is $K$-rational. This completes the proof of \eqref{34}.
\end{proof}

\begin{remark}
 One can also finish the proof by applying the much more general result in \cite[\S 12]{CSS} on intersections of quadrics containing two conjugate skew lines. But the proof of this result is very long
and we therefore found it natural to include the simple Weil restriction argument.
\end{remark}

\end{document}